\documentclass[a4paper]{amsart}
\usepackage{amsmath,amssymb,amsxtra,amsfonts,amsthm,comment,url,microtype}

\usepackage{tikz}
\usetikzlibrary{arrows,backgrounds,decorations,automata}

\theoremstyle{plain}
\newtheorem{theorem}{Theorem}
\newtheorem{lemma}[theorem]{Lemma}
\newtheorem{corollary}[theorem]{Corollary}
\newtheorem{proposition}[theorem]{Proposition}

\theoremstyle{definition}

\begin{document}

\title{Mahler takes a regular view of Zaremba}

\author{Michael Coons}
\address{School of Math.~and Phys.~Sciences\\
University of Newcastle\\
Callaghan\\
Australia}
\email{Michael.Coons@newcastle.edu.au}

\thanks{The research of M.~Coons was partially supported by ARC grant DE140100223.}

\date{\today}

\begin{abstract} In the theory of continued fractions, Zaremba's conjecture states that there is a positive integer $M$ such that each integer is the denominator of a convergent of an ordinary continued fraction with partial quotients bounded by $M$. In this paper, to each such $M$ we associate a regular sequence---in the sense of Allouche and Shallit---and establish various properties and results concerning the generating function of the regular sequence. In particular, we determine the minimal algebraic relation concerning the generating function and its Mahler iterates. 
\end{abstract} 

\maketitle

\section{Introduction}

For all $x\in(0,1)$ we write the ordinary continued fraction expansion of $x$ as $$x=[a_1,a_2,a_3,\ldots]=\cfrac{1}{a_1+\cfrac{1}{a_2+\cfrac{1}{a_3+\cfrac{1}{\ddots}}}},$$ where the positive integers $a_1,a_2,a_3,\ldots,$ are the partial quotients of $x$. The convergents of the number $x$ are the rationals $p_n/q_n:=[a_1,\ldots,a_n]$ and can be computed using the definition or by the well-known relationship $$\left(\begin{matrix} a_n&1\\ 1&0\end{matrix}\right)\cdots\left(\begin{matrix} a_2&1\\ 1&0\end{matrix}\right)\left(\begin{matrix} a_1&1\\ 1&0\end{matrix}\right)=\left(\begin{matrix} q_n&p_{n}\\ q_{n-1}&p_{n-1}\end{matrix}\right).$$ 
See the monograph {\em Neverending Fractions} by Borwein, van der Poorten, Shallit and Zudilin \cite{BvdPSZ} for details and properties regarding continued fractions.

Denote by $\mathcal{B}_k$ the set of real numbers $x\in(0,1)$ all of whose partial quotients are bounded above by $k$. In the early 1970s, Zaremba \cite{Z1972} conjectured that {\em there is a positive integer $k$ such that the set of denominators of the convergents of the elements of $\mathcal{B}_k$ is $\mathbb{N}$.}

While in generality, Zaremba's conjecture remains open, there has been a lot of progress recently. Bourgain and Kontorovich \cite{BK2014} proved that the set of denominators of $\mathcal{B}_{50}$ has full density in $\mathbb{N}$; this was improved by Huang \cite{H2015}, who proved the analogous result for $\mathcal{B}_5$. It is not the purpose of this paper to improve upon the results of Bougain and Kontorvich and Hunag, but to view the denominators of the convergents in a new way using the framework of regular sequences introduced by Allouche and Shallit \cite{AS1992}. 

An integer-valued sequence $\{f(n)\}_{n\geqslant 0}$ is {\em $k$-regular} provided there exist a positive integer $d$, a finite set of matrices $\{{\bf A}_0,\ldots,{\bf A}_{k-1}\}\subseteq \mathbb{Z}^{d\times d}$, and vectors ${\bf v},{\bf w}\in \mathbb{Z}^d$ such that $$f(n)={\bf w}^T {\bf A}_{i_0}\cdots{\bf A}_{i_s} {\bf v},$$ where $(n)_k={i_s}\cdots {i_0}$ is the base-$k$ expansion of $n$; see Allouche and Shallit \cite[Lemma~4.1]{AS1992}. The notion of $k$-regularity is a direct generalisation of automaticity; in fact, a $k$-regular sequence that takes finitely many values can be output by a deterministic finite automaton. We call the generating function $F(z)=\sum_{n\geqslant 0}f(n)z^n$ of a $k$-regular sequence $\{f(n)\}_{n\geqslant 0}$, a {\em $k$-regular function} (or just {\em regular}, when the $k$ is understood). 

We associate the denominators of elements in $\mathcal{B}_k$ to the $k$-regular sequence $\{\kappa(n)\}_{n\geqslant 0}$ defined by \begin{equation}\label{def:kappa}\kappa(n):={\bf w}^T {\bf A}_w {\bf v},\end{equation} where ${\bf w}={\bf v}^T=[1\ 0]$ and for $i=0,1,\ldots,k-1$ $${\bf A}_i:=\left(\begin{matrix} i+1&1\\ 1&0\end{matrix}\right).$$ Here $w\in 1\{0,1\}^*$ corresponds to the reversal of the base-$k$ expansion of $n$; that is, $w=i_0i_1\cdots i_s,$ when $(n)_k=i_s\cdots i_1 i_0.$ The set of values of $\{\kappa(n)\}_{n\geqslant 0}$ are exactly the set of denominators of elements of $\mathcal{B}_k$.

Like all generating functions of regular sequences (see Becker \cite{pgB1994}), the generating function $K(z):=\sum_{\geqslant 0}\kappa(n)z^n$ satisfies a Mahler-type functional equation. Our main results focus on the function $K(z)$. In particular, we start by obtaining the exact functional equation.

\begin{theorem}\label{PkMFE} Let $K(z)$ be as defined above. Then \begin{equation}\label{MFE}K(z)-\left(\sum_{a=0}^{k-1}(a+1)z^a\right)K(z^k)-\left(\sum_{a=0}^{k^2-1}z^a\right)K(z^{k^2})=-\sum_{n=0}^{k-1}z^n.\end{equation} 
\end{theorem}

\noindent In order to study the sequence $\{\kappa(n)\}_{n\geqslant 0}$, we determine the asymptotics of the series $K(z)$ as $z$ radially approaches special points on the unit disk.

\begin{proposition}\label{asymp} Let $m\geqslant 0$ be an integer. As $z\to 1^-$, we have $$K(z^{k^m})=\frac{1}{\alpha_k^{m}}\cdot\frac{C(z)}{(1-z)^{\log_k\alpha_k}} \cdot(1+o(1)),$$ where $C(z)$ is a real-analytic nonzero oscillatory term dependent on $k$, which on the interval $(0,1)$ is bounded away from $0$ and $\infty$, $C(z)=C(z^k)$, and $$\alpha_k:=\frac{k(k+1)+k\sqrt{(k+1)^2+16}}{4}.$$ 
\end{proposition}

The asymptotics of Proposition \ref{asymp} can be used to give results on both the coefficients $\kappa(n)$ and the algebraic character of the function $K(z)$ and its iterates as $z\mapsto z^k$. 

\begin{corollary}\label{cor:asymp} Let $k\geqslant 2$ be an integer, $\{\kappa(n)\}_{n\geqslant 0}$ be the $k$-regular sequence defined in \eqref{def:kappa}, and $\alpha_k$ be as given in Proposition \ref{asymp}. Then there are positive constants $c_1=c_1(k)$ and $c_2=c_2(k)$ such that as $N\to \infty$, $$c_1\leqslant \frac{1}{n^{\log_k\alpha_k}}\sum_{n\leqslant N}\kappa(n) \leqslant c_2.$$
\end{corollary}

\noindent In fact, one can be much more specific concerning the behaviour of the sums in Corollary~\ref{cor:asymp}---we will explore these details in a later section.

The asymptotics of Proposition \ref{asymp} can be used to show that the function $K(z)$ is unbounded as $z$ radially approaches each root of unity of order $k^m$ ($m\geqslant 0$); see Proposition \ref{Omega}. Combining this with a recent result of Coons and Tachiya \cite{CTbams} gives the following proposition concerning the algebraic character of the function $K(z)$.

\begin{proposition} Let $G(z)$ be any meromorphic complex-valued function. For each integer $k\geqslant 2$, the functions $K(z)$ and $G(z)$ are algebraically independent over $\mathbb{C}(z)$. In particular, the function $K(z)$ is transcendental over $\mathbb{C}(z)$.
\end{proposition}

Indeed, Proposition \ref{asymp} can be used to give additional information about the function $K(z)$ and its iterates under the map $z\mapsto z^k$. We use Proposition \ref{asymp} to prove the following statement.

\begin{theorem}\label{algindPk} For each integer $k\geqslant 2$, the functions $K(z)$ and $K(z^k)$ are algebraically independent over $\mathbb{C}(z)$.
\end{theorem}

\noindent In some sense, Theorem \ref{algindPk} can be interpreted to say that the Mahler functional equation given in Theorem \ref{PkMFE} is essentially the minimal algebraic relation between the function $K(z)$ and its Mahler iterates, that is, the set $\{K(z),K(z^k),K(z^{k^2}),\ldots\}.$

Note that the combination of Theorem \ref{algindPk} with the celebrated result of Ku.~Nishioka \cite[Corollary~2]{N1990} gives the following result on the algebraic independence of certain special values of $K(z)$ and $K(z^k)$. 

\begin{corollary} Let $k\geqslant 2$ be an integer and $\alpha$ be a nonzero algebraic number in the open unit disc. Then $K(\alpha)$ and $K(\alpha^k)$ are algebraically independent over $\mathbb{Q}$.
\end{corollary}

This paper is outlined as follows. In Section \ref{sec:Mahlfunc} prove Theorem \ref{PkMFE}. Section \ref{sec:radas} contains the results leading to the radial asymptotics described in Proposition \ref{asymp}, which are then used in Section \ref{sec:algindep} to prove Theorem \ref{algindPk}. The final section contains a brief discussion of Corollary \ref{cor:asymp} and then a question about a possible connection between the function $\kappa(n)$ (specialised at $k=2$) and the Takagi function.

\section{A Mahler-type functional equation for $K(z)$}\label{sec:Mahlfunc}

In this section, we provide the proof of Theorem \ref{PkMFE}. This theorem follows from recurrences satisfied by the sequence of values of $\kappa$, which we record in the following lemma. 

\begin{lemma}\label{quadn} If $n\geqslant 1$ and $a,b\in\{0,\ldots,k-1\}$, then $$\kappa(k^2n+kb+a)=(a+1)\kappa(kn+b)+\kappa(n).$$
\end{lemma}

\begin{proof} For any $i\in\{0,\ldots,k-1\}$ we have both $$[0\ 1]{\bf A}_i=[1\ 0]\quad\mbox{and}\quad [1\ 0]{\bf A}_i=(i+1)[1\ 0]+[0\ 1].$$ If $n\geqslant 1$ and $a,b\in\{0,\ldots,k-1\}$, then writing $w$ as the reversal of $(n)_k$ we have \begin{align*}\kappa(k^2n+kb+a)&=[1\ 0]{\bf A}_a{\bf A}_b{\bf A}_w [1\ 0]^T\\
&=\left((a+1)[1\ 0]+[0\ 1]\right){\bf A}_b{\bf A}_w [1\ 0]^T\\
&=(a+1)[1\ 0]{\bf A}_b{\bf A}_w [1\ 0]^T+[0\ 1]{\bf A}_b{\bf A}_w [1\ 0]^T\\
&=(a+1)\kappa(kn+b)+[1\ 0]{\bf A}_w [1\ 0]^T\\
&=(a+1)\kappa(kn+b)+\kappa(n).\qedhere
\end{align*}
\end{proof}

\begin{proof}[Proof of Theorem \ref{PkMFE}] We use the relationship in Lemma \ref{quadn} to give 
\begin{align*}
K(z)
&= \sum_{n\geqslant 1}\sum_{i=0}^{k-1}\sum_{j=0}^{k-1} \kappa(k^2n+ki+j)z^{k^2n+ki+j}+\sum_{n=0}^{k^2-1}\kappa(n)z^n\\
&= \sum_{n\geqslant 1}\sum_{i=0}^{k-1}\sum_{j=0}^{k-1} \left[(j+1)\kappa(kn+i)+\kappa(n)\right]z^{k^2n+ki+j}+\sum_{n=0}^{k^2-1}\kappa(n)z^n\\
&= \sum_{j=0}^{k-1}(j+1)z^j\sum_{n\geqslant 1}\sum_{i=0}^{k-1}\kappa(kn+i)(z^k)^{kn+i}\\
&\qquad\qquad+\sum_{j=0}^{k-1}z^j\sum_{i=0}^{k-1}z^{ki}\sum_{n\geqslant 1}\kappa(n)z^{k^2n}+\sum_{n=0}^{k^2-1}\kappa(n)z^n,
\end{align*} so that
\begin{align*}
K(z)&= \sum_{j=0}^{k-1}(j+1)z^j\left(K(z^k)-\sum_{n=0}^{k-1}\kappa(n)z^{kn}\right)\\
&\qquad\qquad+\left(\sum_{j=0}^{k^2-1}z^j\right)\left(K(z^{k^2})-\kappa(0)\right)+\sum_{n=0}^{k^2-1}\kappa(n)z^n.
\end{align*}
Rearranging this, we have 
\begin{multline}
K(z)-\left(\sum_{j=0}^{k-1}(j+1)z^j\right)K(z^k)-\left(\sum_{j=0}^{k^2-1}z^j\right)K(z^{k^2})\\ \label{RHSq}=\sum_{n=0}^{k^2-1}\kappa(n)z^n-\kappa(0)\left(\sum_{j=0}^{k^2-1}z^j\right)-\left(\sum_{j=0}^{k-1}(j+1)z^j\right)\left(\sum_{n=0}^{k-1}\kappa(n)z^{kn}\right).
\end{multline} Using the matrix representation we have $\kappa(n)=n+1$ for $n=0,\ldots,k-1$, and for $n=ak+b$ with $a=1,\ldots,k-1$ and $b=0,\ldots,k-1$ we have $\kappa(n)=(a+1)(b+1)+1$. Thus the continuing the equality \eqref{RHSq} 
\begin{align*}
K(&z)-\left(\sum_{j=0}^{k-1}(j+1)z^j\right)K(z^k)-\left(\sum_{j=0}^{k^2-1}z^j\right)K(z^{k^2})\\
&= \sum_{n=k}^{k^2-1}\kappa(n)z^n+\sum_{n=0}^{k-1}\kappa(n)z^n-\left(\sum_{j=0}^{k^2-1}z^j\right)-\left(\sum_{j=0}^{k-1}(j+1)z^j\right)\left(\sum_{n=0}^{k-1}(n+1)z^{kn}\right)\\
&= \sum_{a=1}^{k-1}\sum_{b=0}^{k-1}(a+1)(b+1)z^{ak+b}+\sum_{n=k}^{k^2-1}z^n+\sum_{n=0}^{k-1}(n+1)z^n-\left(\sum_{j=0}^{k^2-1}z^j\right)\\
&\qquad\qquad-\sum_{j=0}^{k-1}(j+1)z^j-\sum_{j=0}^{k-1}\sum_{n=1}^{k-1}(j+1)(n+1)z^{kn+j},
\end{align*} which, after cancelling terms, proves the theorem.
\end{proof}

\section{Radial asymptotics of $K(z)$}\label{sec:radas}

We prove Theorem \ref{asymp} in this section by appealing to a recent result of Bell and Coons \cite{BC2017}, which provides the radial asymptotics of a Mahler function $F(z)$ based on the existence of the Mahler eigenvalue associated to the function $F(z)$.

\begin{proof}[Proof of Proposition \ref{asymp}] We start with the functional equation \eqref{MFE} established in Theorem \ref{PkMFE} and send $z$ to $z^k$ to get $$K(z^k)-\left(\sum_{a=0}^{k-1}(a+1)z^{ka}\right)K(z^{k^2})-\left(\sum_{a=0}^{k^2-1}z^{ka}\right)K(z^{k^3})=q(z^k),$$ where we have set $$q(z)=q_k(z):=-\sum_{n=0}^{k-1}z^n.$$ We now multiply the original functional equation \eqref{MFE} by $q(z^k)$ and the new functional equation by $q(z)$ and subtract the resulting functional equations to get the homogeneous equation \begin{multline*}q(z^k)K(z)-\left[q(z^k)\left(\sum_{a=0}^{k-1}(a+1)z^{a}\right)+q(z)\right]K(z^{k})\\-\left[q(z^k)\left(\sum_{a=0}^{k^2-1}z^{a}\right)-q(z)\left(\sum_{a=0}^{k-1}(a+1)z^{ka}\right)\right]K(z^{k^2}) \\+q(z)\left(\sum_{a=0}^{k^2-1}z^{ka}\right)K(z^{k^3})=0.\end{multline*}

Following the method of Bell and Coons \cite{BC2017} (see also Brent, Coons, and Zudilin \cite{BCZ2016}) we use the homogeneous functional equation to form the characteristic polynomial \begin{align*}\chi_{K}(\lambda)&=q(1)\left[\lambda^3-\left(\sum_{a=0}^{k-1}(a+1)+1\right)\lambda^2-\left(\sum_{a=0}^{k^2-1}1-\sum_{a=0}^{k-1}(a+1)\right)\lambda+\sum_{a=0}^{k^2-1}1\right]\\
&=-k(\lambda-1)\left[\lambda^2-\frac{k(k+1)}{2}\lambda-k^2\right],\end{align*} where we have used the value of $q(1)=-k$. Since $\chi_{K}(\lambda)$ has three distinct roots for each positive integer $k\geqslant 2$, and the sequence $\{\kappa(n)\}_{n\geqslant 0}$ grows super-linearly, {Theorem~1} of Bell and Coons \cite{BC2017} applies to give that as $z\to 1^-$ $$K(z)=\frac{C(z)}{(1-z)^{\log_k \alpha_k}} (1+o(1)),$$ where $\log_k$ denotes the principal value of the base-$k$ logarithm, $C(z)$ is a real-analytic nonzero oscillatory term dependent on $k$, which on the interval $(0,1)$ is bounded away from $0$ and $\infty$, and satisfies $C(z)=C(z^k)$ and $$\alpha_k=\frac{k(k+1)+k\sqrt{(k+1)^2+16}}{4}.$$

It remains to show that these asymptotics hold for $z$ replaced by $z^{k^m}$ for any positive integer $m$. This follows exactly from the fact that $C(z^{k^m})=C(z)$ and the identity $(1-z^{k^m})^{\log_k \alpha_k}=(1-z)^{\log_k \alpha_k}\alpha_k^{m}(1+o(1))$ as $z\to 1^-$.
\end{proof}

Using Proposition \ref{asymp} we can determine the asymptotics of $K(z)$ as $z$ radially approaches any $k^n$th root of unity.

\begin{proposition}\label{Omega} Let $C(z)$ and $\alpha_k$ be as defined in Theorem \ref{asymp}, $m\geqslant 0$ be an integer, and $\xi$ be a root of unity of degree $k^m$. As $z\to 1^-$, we have $$K(\xi z)=\Omega(\xi)\cdot\frac{C(z)}{(1-z)^{\log_k\alpha_k}} \cdot(1+o(1)),$$ where the function $\Omega(z)$ satisfies $\Omega(1)=1$ and $$\left(z-1\right)\Omega(z)=\left((k+2)z^k+\sum_{a=0}^{k-1}z^a\right)\alpha_k^{-1}\Omega(z^k)+\left(z^{k^2}-1\right)\alpha_k^{-2}\Omega(z^{k^2}).$$
\end{proposition}

\begin{proof} It is clear from Proposition \ref{asymp} that $\Omega(1)=1$. Using the functional equation \eqref{MFE} and Proposition \ref{asymp}, for any $k$th root of unity $\xi_1$, we have as $z\to 1^-$
\begin{align*}
K(\xi_1 z)&=\left(\sum_{a=0}^{k-1}(a+1)(\xi_1z)^a\right)K(z^k)+\left(\sum_{a=0}^{k^2-1}(\xi_1z)^a\right)K(z^{k^2})+q(\xi_1z)\\
&=\left(\sum_{a=0}^{k-1}(a+1)\xi_1^a\right)\frac{1}{\alpha_k}\cdot\frac{C(z)}{(1-z)^{\log_k\alpha_k}} \cdot(1+o(1))\\
&\qquad\qquad+\left(\sum_{a=0}^{k^2-1}\xi_1^a\right)\frac{1}{\alpha_k^{2}}\cdot\frac{C(z)}{(1-z)^{\log_k\alpha_k}} \cdot(1+o(1))+q(\xi_1)(1+o(1))\\
&=\left[\left(\sum_{a=0}^{k-1}(a+1)\xi_1^a\right)\frac{1}{\alpha_k}+\left(\sum_{a=0}^{k^2-1}\xi_1^a\right)\frac{1}{\alpha_k^{2}}\right]\frac{C(z)}{(1-z)^{\log_k\alpha_k}} \cdot(1+o(1))\\
&=\Omega(\xi_1)\cdot\frac{C(z)}{(1-z)^{\log_k\alpha_k}} \cdot(1+o(1)).
\end{align*}

Now if $\xi_2$ is any root of unity of degree $k^2$ such that $\xi_2^k=\xi_1$, then 
\begin{align*}
K(\xi_2 z)&=\left(\sum_{a=0}^{k-1}(a+1)\xi_2^a\right)\frac{\Omega(\xi_1)}{\alpha_k}\cdot\frac{C(z)}{(1-z)^{\log_k\alpha_k}} \cdot(1+o(1))\\
&\qquad\qquad+\left(\sum_{a=0}^{k^2-1}\xi_2^a\right)\frac{1}{\alpha_k^{2}}\cdot\frac{C(z)}{(1-z)^{\log_k\alpha_k}} \cdot(1+o(1))+q(\xi_2)(1+o(1)),
\end{align*} so that
\begin{align*}
K(\xi_2 z)&=\left[\left(\sum_{a=0}^{k-1}(a+1)\xi_2^a\right)\frac{\Omega(\xi_1)}{\alpha_k}+\left(\sum_{a=0}^{k^2-1}\xi_2^a\right)\frac{1}{\alpha_k^{2}}\right]\frac{C(z)}{(1-z)^{\log_k\alpha_k}} \cdot(1+o(1))\\
&=\Omega(\xi_2)\cdot\frac{C(z)}{(1-z)^{\log_k\alpha_k}} \cdot(1+o(1)),
\end{align*} where $$\Omega(\xi_2)=\left(\sum_{a=0}^{k-1}(a+1)\xi_2^a\right)\frac{\Omega(\xi_1)}{\alpha_k}+\left(\sum_{a=0}^{k^2-1}\xi_2^a\right)\frac{1}{\alpha_k^{2}}.$$ Continuing in this manner defines a function $\Omega(\xi)$ on the roots of unity of degree $k^n$ for any $n\geqslant 0$ such that $$\Omega(\xi)=\left(\sum_{a=0}^{k-1}(a+1)\xi^a\right)\frac{\Omega(\xi^k)}{\alpha_k}+\left(\sum_{a=0}^{k^2-1}\xi^a\right)\frac{\Omega(\xi^{k^2})}{\alpha_k^{2}}.$$ Combining this with the identities $$\sum_{a=0}^{k-1}(a+1)z^a=\frac{1}{z-1}\left((k+2)z^k+\sum_{a=0}^{k-1}z^a\right)\quad\mbox{and}\quad \sum_{a=0}^{k^2-1}z^a =\frac{z^{k^2}-1}{z-1}$$ provides the result.
\end{proof}

Note that the function $\Omega(z)$ depends on $k$ and is defined only on the set of roots of unity of degree $k^m$ for integers $m\geqslant 0$. 

\section{Algebraic independence of $K(z)$ and $K(z^k)$}\label{sec:algindep}

In this section, we prove Theorem \ref{algindPk}. It follows from the following general result and two lemmas; see Brent, Coons and Zudilin \cite[Theorem~4]{BCZ2016} for a similar argument.

\begin{proposition}\label{nonrelation} Let $k\geqslant 2$ be a positive integer and suppose that there is a rational function $\lambda(z)$  and polynomials $p_0(z),\cdots,p_M(z)\in\mathbb{C}[z]$ such that \begin{equation}\label{T4} \lambda(z)\sum_{m=0}^M p_m(z) y^m = \sum_{m=0}^M p_{M-m}(z^k)\left((z-1)(z^k-1)y+s(z)\right)^{m},\end{equation} where $$s(z):=-(k+2)z^k-\sum_{a=0}^{k-1}z^a.$$ Then $p_j(z)=0$ for each $j=0,\ldots,M$.
\end{proposition}

\begin{proof} Assume to the contrary, that a nontrivial collection of polynomials $p_0(z),$ $p_1(z),\ldots,p_M(z)$ satisfying \eqref{T4} exists. If the greatest common divisor of the polynomials is $p(z)$ then dividing them all by $p(z)$ we arrive at the relation \eqref{T4} for the newer normalised polynomials, but with $\lambda(z)$ replaced by $\lambda(z)p(z)/p(z^k)$. Therefore, we can assume without loss of generality that the polynomials $p_{i}(z)$ in \eqref{T4} are relatively prime. Moreover, without loss of generality we can assume that the existing polynomials all have \emph{rational} coefficients, since the identity \eqref{T4} is over the field of rationals, so that each $p_{i}(z)\in\mathbb{Q}[z]$ and $\lambda(z)\in\mathbb Q(z)$.

Assuming $\lambda(z)$ is nonzero, write $\lambda(z)=a(z)/b(z)$, where $\gcd(a(z),b(z))=1$, so that \eqref{T4} becomes
\begin{equation}
\label{ab}
a(z)\sum_{m=0}^M p_m(z) y^m = b(z)\sum_{m=0}^M p_{M-m}(z^k)\left((z-1)(z^k-1)y+s(z)\right)^{m}.
\end{equation}
It follows immediately that any polynomial $p_m(z)$ on the left-hand side of \eqref{ab} is divisible by $b(z)$, hence $b:=b(z)$ is a constant.
By substituting $x=(z-1)(z^k-1)y+s(z)$, we write \eqref{ab} as
\begin{multline*}
a(z)\sum_{m=0}^M p_m(z)\left[(z-1)(z^k-1)\right]^{M-m}(x-s(z))^m\\
=b\left[(z-1)(z^k-1)\right]^M\sum_{m=0}^M p_{M-m}(z^k)x^{m},
\end{multline*}
from which we conclude that each $p_m(z^k)$ is divisible by $a(z)/\left[(z-1)(z^k-1)\right]^{N}$ where $N$ is the highest power of $\left[(z-1)(z^k-1)\right]$ dividing $a(z)$. Since $\gcd(p_0(z^k),$ $\dots,p_M(z^k))=1$, we have that $a:=a(z)/\left[(z-1)(z^k-1)\right]^{N}$ is a constant. In summary, $$\lambda(z)=\lambda \left[(z-1)(z^k-1)\right]^{N}$$ for some $\lambda\in\mathbb Q$ and $N\in\mathbb{Z}_{\geqslant 0}$; that is,
\begin{multline}
\label{lam}
\lambda \left[(z-1)(z^k-1)\right]^{N}\sum_{m=0}^M p_m(z) y^m \\ =\sum_{m=0}^M p_{M-m}(z^k)\left((z-1)(z^k-1)y+s(z)\right)^{m}.
\end{multline}
Note that the constant $\lambda$ must be nonzero since otherwise, by substituting $z=0$ into \eqref{lam}, and noting that $s(0)=1$ for any choice of $k$, each of the $p_m(z^k)$ would have the common divisor $z$.

If we iterate the righthand side of \eqref{lam} one more time we arrive at the identity
\begin{multline}
\label{lam2}
\lambda \left[(z-1)(z^k-1)^2(z^{k^2}-1)\right]^{N}\sum_{m=0}^M p_m(z) y^m \\ =\sum_{m=0}^M p_{m}(z^{k^2})\left((z-1)(z^k-1)^2(z^{k^2}-1)y\right.\\ +\left.(z^k-1)(z^{k^2}-1)s(z)+s(z^k)\right)^{m}.
\end{multline}
The coefficients of $y^M$ on each side of \eqref{lam2} are equal. That is,
\begin{multline}
\label{NM} 
\lambda \left[(z-1)(z^k-1)^2(z^{k^2}-1)\right]^{N}p_M(z)
\\ =p_{M}(z^{k^2})\left[(z-1)(z^k-1)^2(z^{k^2}-1)\right]^M.
\end{multline}
The power of the factor $z-1$ is the same in both polynomials $p_M(z)$ and $p_M(z^{k^2})$, so that comparing the powers of the factor $z-1$ on both sides of \eqref{NM} we deduce that $N=M$. Further, comparing degrees in \eqref{NM}, we have that $p_M(z)$ is a constant and then, necessarily, $\lambda=1$. Without loss of generality, we may assume that $p_M(z)=1$ (we just shift the constant to the other polynomials by dividing if needed). 

With this information, we continue by equating the coefficients of $y^{M-1}$ on each side of \eqref{lam2} to give  
\begin{multline*}
\left[(z-1)(z^k-1)^2(z^{k^2}-1)\right]^{M}p_{M-1}(z)
\\ =p_{M}(z^{k^2})\left[(z-1)(z^k-1)^2(z^{k^2}-1)\right]^{M-1}\left((z^k-1)(z^{k^2}-1)s(z)+s(z^k)\right)\\
+p_{M-1}(z^{k^2})\left[(z-1)(z^k-1)^2(z^{k^2}-1)\right]^{M-1},
\end{multline*} 
which reduces to 
\begin{multline}
\label{M-1} 
(z-1)(z^k-1)^2(z^{k^2}-1)p_{M-1}(z)
\\ =\left((z^k-1)(z^{k^2}-1)s(z)+s(z^k)\right)
+p_{M-1}(z^{k^2}).
\end{multline} 

Finally, evaluating \eqref{M-1} at $z=0$ gives that $s(0)=0$, contradicting the fact that $s(0)=-1$. This completes the proof of the theorem.
\end{proof}

\begin{lemma}\label{omegatrans} The function $$\omega(z):=\frac{\alpha_k\Omega(z)}{\Omega(z^k)}$$ is transcendental over $\mathbb{C}(z)$.
\end{lemma}

\begin{proof} Using the functional equation for $\Omega(z)$ from Theorem \ref{Omega} we have $$\left(z-1\right)\omega(z)=(k+2)z^k+\sum_{a=0}^{k-1}z^a+\frac{z^{k^2}-1}{\omega(z^{k})}.$$ Assume to the contrary that $\omega(z)$ is algebraic, so that $\omega(z)/(z^k-1)$ is also algebraic. Then there is a nontrivial relation, over the set of roots of unity of degree $k^m$ for nonnegative integers $n$, \begin{equation}\label{y}\sum_{m=0}^M p_m(z) \left(\frac{\omega(z)}{z^k-1}\right)^m=\left.\sum_{m=0}^M p_m(z) y^m\right|_{y=\omega(z)/(z^k-1)}=0,\end{equation} where here we take $M$ to be the smallest such positive integer. Multiplying the relation by $((z^k-1)/\omega(z))^M$ we have $$\sum_{m=0}^M p_{M-m}(z) \left(\frac{z^k-1}{\omega(z)}\right)^m=\left.\sum_{m=0}^M p_{M-m}(z) (y^{-1})^m\right|_{y=\omega(z)/(z^k-1)}=0,$$ whereby sending $z\mapsto z^k$ and using the functional equation for $\omega(z)$, we have \begin{multline}\label{yzk}\sum_{m=0}^M p_{M-m}(z^k)\left(\left(z-1\right) \omega(z) -s(z)\right)^m\\ =\left.\sum_{m=0}^M p_{M-m}(z^k)\left[(z-1)(z^k-1)y -s(z)\right)^m\right|_{y={\omega(z)}/{(z^k-1)}}=0.\end{multline} If the two algebraic relations in \eqref{y} and \eqref{yzk} are not proportional, then a suitable linear combination of the two will eliminate the term $y^M$ and result in a nontrivial algebraic relation for $\omega(z)$ of degree smaller than $M$, resulting in a contradiction. On the other hand, the proportionality is not possible in view of proposition \ref{nonrelation}. This proves the lemma.
\end{proof}

\begin{lemma}\label{lem2} Let $k\geqslant 2$ be a positive integer and suppose that there are polynomials $p_0(z),\ldots,p_M(z)\in\mathbb{C}[z]$ such that $$\sum_{m=0}^M p_m(\xi) \Omega(\xi)^m\Omega(\xi^k)^{M-m}=0,$$ for any root of unity of degree $k^n$ for a nonnegative integer $n$. Then $p_m(z)=0$ for each $m=0,\ldots,M$.
\end{lemma}

\begin{proof} If a relationship as stated in the lemma exists, then it implies, using the function $\omega(z)$, that there is a relationship $$\sum_{m=0}^M p_m(z)\alpha_k^{-m}\omega(z)^m=0$$ on the set of roots of unity of degree $k^n$ for nonnegative integers $n$. But this is not possible because of the transcendence of $\omega(z)$ established by Lemma \ref{omegatrans}.
\end{proof}

\begin{proof}[Proof of Theorem \ref{algindPk}] For the sake of a contradiction, assume that the theorem is false and that we have an algebraic relation
\begin{equation}
\sum_{(m_0,m_1)\in{\bf M}}p_{m_0,m_1}(z)K(z)^{m_0}K(z^k)^{m_1}=0,
\label{SF}
\end{equation}
where the set ${\bf M}$ of multi-indices $(m_0,m_1)\in\mathbb Z_{\geqslant 0}^2$
is finite and none of the polynomials $p_{m_0,m_1}(z)$ in the sum is identically zero. Without loss of generality, we can assume that the polynomial $\sum_{(m_0,m_1)\in{\bf M}}p_{m_0,m_1}(z)y_0^{m_0}y_1^{m_1}$ of three variables is irreducible.

Note that as $z\to 1^-$ along a sequence of numbers $z_0^{-k^n}$, we have 
\begin{equation*}
K(z)^{m_0}K(z^k)^{m_1}
=C_{m_0,m_1}\cdot\frac{\Omega(\xi)^{m_0}\Omega(\xi^k)^{m_1}}{
(1-z)^{(\log_k\alpha_k)(m_0+m_1)}}(1+o(1))
\end{equation*}
where $$
C_{m_0,m_1}:=C^{m_0+m_1}\alpha_k^{m_1},
$$ and $C=C(e^{-z_0/k})$ does not depend on $\xi$ or $z$, the latter chosen along the sequence of numbers $z_0^{-k^n}$.

Denote by ${\bf M'}$ the subset of all multi-indices of ${\bf M}$ for which the quantity $$\beta:=(\log_k\alpha_k)(m_0+m_1)$$ is maximal; in particular, $(m_0+m_1)$ is the same for all $(m_0,m_1)\in{\bf M'}$.

Multiplying all the terms in the sum \eqref{SF} by $(1-z)^\beta$ and letting $z\to1^-$, we deduce that
\begin{equation}
\sum_{(m_0,m_1)\in{\bf M'}}C_{m_0,m_1}\cdot p_{m_0,m_1}(\xi)\cdot\Omega(\xi)^{m_0}\Omega(\xi^4)^{m_1}=0.
\label{SF1}
\end{equation}
for any root of unity $\xi$ under consideration.
But since $m_0+m_1=M$ is constant for each $(m_0,m_1)\in{\bf M'}$, the Equation \eqref{SF1} becomes
\begin{equation*}
\sum_{(m_0,M-m_0)\in{\bf M'}}C_{m_0,M-m_0}\cdot p_{m_0,M-m_0}(\xi)\cdot\Omega(\xi)^{m_0}\Omega(\xi^4)^{M-m_0}=0
\end{equation*}
for any root of unity $\xi$ of degree $k^n$.
By Lemma~\ref{lem2}, this is only possible when $p_{m_0,M-m_0}(z)=0$ identically, a contradiction to our choice of ${\bf M}$.
\end{proof}

\section{Variations and musings on Corollary \ref{cor:asymp}}\label{sec:varcor}

Corollary \ref{cor:asymp} follows directly from a very recent result of ours; see the proof of Theorem 2 in \cite{Cmahleig} and note therein that in this case $m_\kappa=0$ since the associated eigenvalues are all nonzero and distinct.

For the remainder of this section, we focus on the case $k=2$ of $\kappa(n)$. To highlight this, we will use the slightly modified notation $\kappa_2(n)$. Using the matrix definition of $\kappa_2(n)$ one can quickly see that both $$\max_{a\in[2^{n-1},2^{n})}\kappa_2(a)=\kappa_2(2^{n}-1)$$
 and $$\limsup_{n\to\infty}\frac{\kappa_2(n)}{n^{\log_2 (1+\sqrt{2})}}=\frac{2+\sqrt{2}}{4},$$ where as before $\log_2$ indicates the base-$2$ logarithm.
 
The graph of the function $\kappa_2(n)$ is not so enlightening, but the graph of the partial sums is an entirely different matter. Recall that Corollay \ref{cor:asymp} give that the partial sums $\sum_{n\leqslant N}\kappa_2(n)$ are bounded above and below by constant multiples of $N^2$. It is also quite clear, using the theory of regular sequences, that the values of $N^{-2}\sum_{n\leqslant N}\kappa_2(n)$ are periodic between powers of $2$---see Figure \ref{k2sum1516} for a graph of the period in the interval $[2^{15},2^{16})$.

\begin{figure}[htp]
\includegraphics[width=4.7in,height=2.5in]{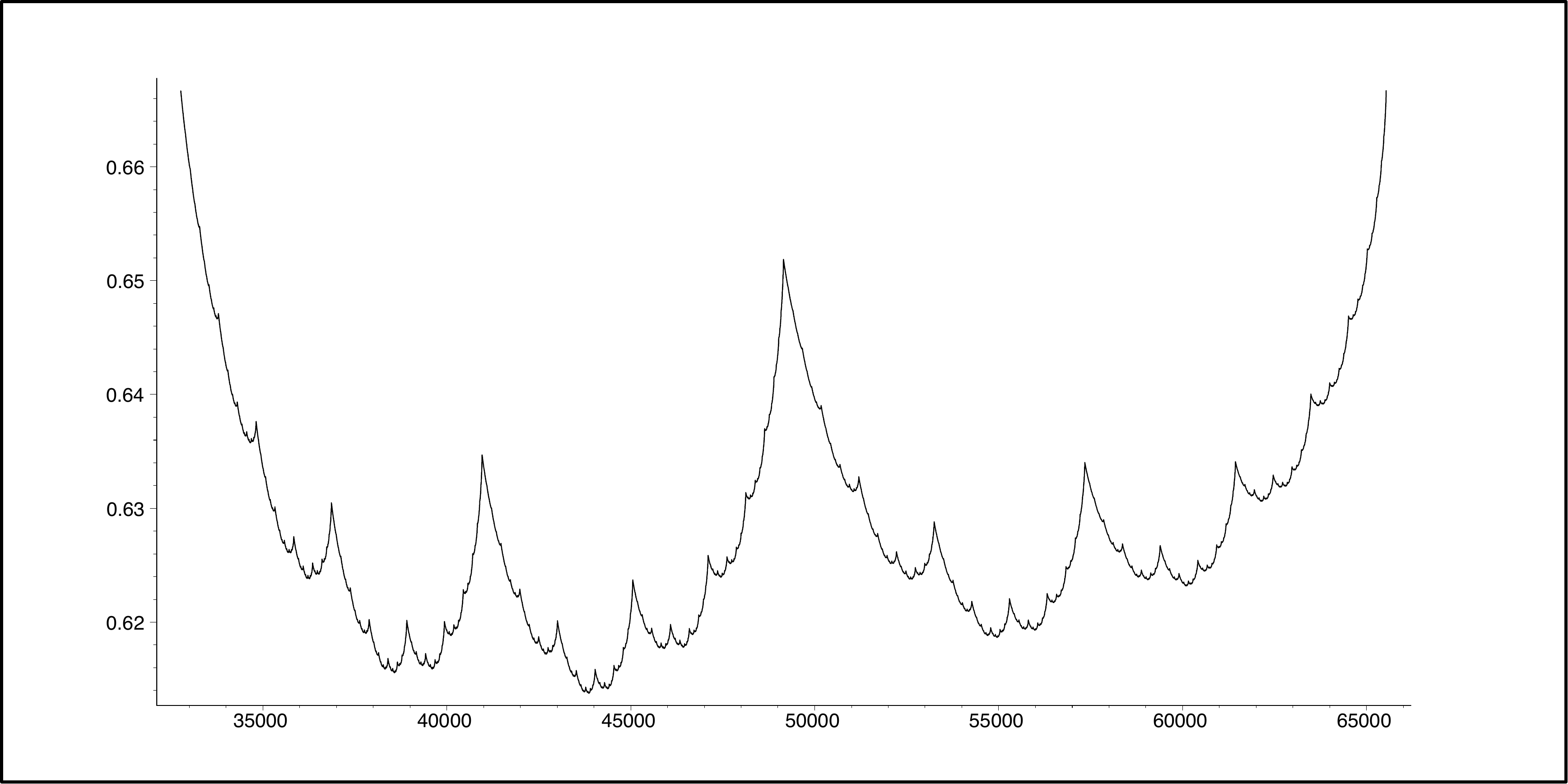}
\caption{The values of $N^{-2}\sum_{n\leqslant N}\kappa_2(n)$ (for $k=2$) with $N\in[2^{15}, 2^{16}]$.}
\label{k2sum1516}
\end{figure}

Figure \ref{k2sum1516} elicits---at least for this author---an immediate feeling of {\em d\'ej\`a vu}. Indeed, scouring back through the literature, one comes across the Takagi function, defined on the unit interval $[0,1]$ by $$\tau(x):=\sum_{n\geqslant 0}\frac{1}{2^n}\|2^nx\|,$$ where $\|y\|$ denotes the distance from $y$ to the nearest integer. For a detailed survey of the Takagi function see the paper of Lagarias \cite{L2012}. The Takagi function is plotted in Figure \ref{takagi}.

\begin{figure}[htp]
\includegraphics[width=4.7in,height=2.5in]{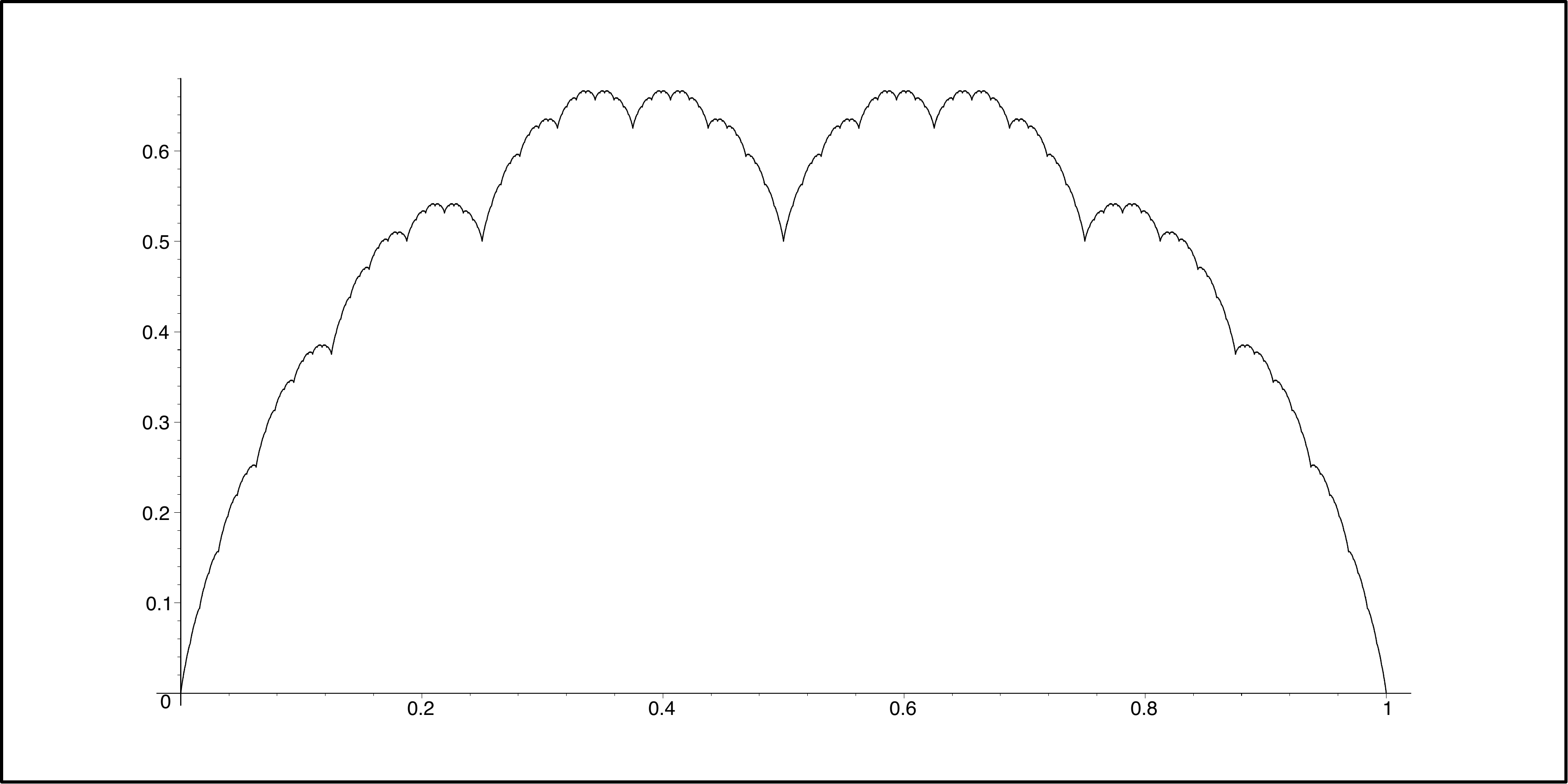}
\caption{The Takagi function $\tau(x)$.}
\label{takagi}
\end{figure}

An explicit connection between these two pictures (and functions) is not immediately evident. We leave the reader with the very vague question of determining if there is a describable relationship between these functions. We find the similarities too compelling to not warrant further study.\\

\noindent{\bf Acknowledgements.} Parts of this paper during visits to both the Erwin Schr\"odinger Institute  (ESI) for Mathematics and Physics during the workshop on ``Normal Numbers: Arithmetic, Computational and Probabilistic Aspects'' and the Alfr\'ed R\'enyi Institute of Mathematics of the Hungarian Academy of Sciences. The author thanks both the ESI and the R\'enyi Institute for their support. The author also thanks Wadim Zudilin for many stimulating conversations and constructive suggestions.

\bibliographystyle{amsplain}
\providecommand{\bysame}{\leavevmode\hbox to3em{\hrulefill}\thinspace}
\providecommand{\MR}{\relax\ifhmode\unskip\space\fi MR }
\providecommand{\MRhref}[2]{%
  \href{http://www.ams.org/mathscinet-getitem?mr=#1}{#2}
}
\providecommand{\href}[2]{#2}


\end{document}